\documentclass[10pt]{amsart}
\usepackage{amsmath,amsthm,amssymb}
\usepackage{latexsym}
\usepackage{graphicx,epsfig,color}

\def\rr{\mathbb R}

\usepackage{color}

\newtheorem{theoreme}{Theorem}[section]
\newtheorem{proposition}[theoreme]{Proposition}

\newtheorem{corollary}[theoreme]{Corollary}

\newtheorem{rem}[theoreme]{Remark}

\newtheorem{3.10}[theoreme]{``Theorem"}
\newtheorem{definition}[theoreme]{Definition}

\theoremstyle{remark}

\numberwithin{equation}{section}

\begin{document}

\date{\today}


\title[Uniform  synchronization] {Uniform  synchronization  of an abstract  linear  second order evolution system}

\maketitle

\centerline {Tatsien Li\footnote {\em School of Mathematical Sciences, Fudan University, Shanghai 200433, China;  Shanghai Key Laboratory for Contemporary Applied Mathematic;  Nonlinear Mathematical Modeling and Methods Laboratory, dqli@fudan.edu.cn.} \quad Bopeng  Rao\footnote{\em Corresponding author. Institut de Recherche Math\'ematique Avanc\'ee, Universit\'e de Strasbourg, 67084  Strasbourg, France, bopeng.rao@math.unistra.fr}}

\begin{abstract}   Although the mathematical study on the synchronization of wave equations at  finite horizon has been  well developed, there was few results on the   synchronization  of wave equations  for long-time horizon.  The aim of the paper is to investigate the  uniform synchronization at  the infinite  horizon for
one abstract linear second order evolution system in a Hilbert space.

First, using the classical compact perturbation theory on the uniform stability of semigroups of contractions, 
we will establish a lower bound on the number of 
damping, necessary for the uniform synchronization of  the considered  system.   Then, under the minimum number of damping,  we 
clarify the algebraic structure of the system as well as 
the necessity of the conditions of compatibility on the coupling matrices.  We then establish  the uniform synchronization by the compact perturbation  method and  then give   the  dynamics  of the  asymptotic orbit.  Various  applications are given for  the system of wave equations with boundary  feedback or (and) locally distributed feedback,  and for  the system of  Kirchhoff plate with  distributed feedback. 
Some  open questions  are raised  at the end of the paper for future development.  

The study is based on the synchronization theory and the compact perturbation of semigroups.  
\end{abstract}

\bigskip


\bigskip 

{\bf MS Classification 2010}\quad  93B05,  93C20, 35L53

\vskip 1cm


\section{Introduction} \label{s1}
Synchronization is a widespread natural phenomenon. It was first observed by Huygens in 1665 \cite{Huygens}. The theoretical research on synchronization from the  mathematical point of view dates back to Wiener in 1950s  in \cite{wiener} (Chapter 10). Since 2012,  Li and Rao started the research on the synchronization  in a finite time for a coupled system  of wave equations  with  Dirichele   boundary controls \cite{Raofr, {Rao3}}).  Later,   the  synchronization has been carried out for a coupled system  of wave equations with various boundary controls, the most part of their results was  recently collected in the monograph  \cite{book}.  The optimal control for the exact synchronization of parabolic system was   recently investigated by Wang and Yan  in \cite{wang}.
Consequently, the study of synchronization becomes a part of research in control theory. 

In a recent work \cite{rao11, rao2},   we showed that   under   Kalman's rank condition, the observability of a scalar equation  implies  the uniqueness  of solution  to a system  of elliptic operators. Using  this result,  we   have established  the asymptotic synchronization by groups for 
 second order evolution systems.  

The objective of this work  is to investigate the uniform  synchronization for second order evolution systems.  Let us briefly describe the  formulation and the main ideas. 

Let $ H$ and $ V $ be  two separated  Hilbert spaces such that
$ V \subset  H \subset V', $ $V'$ being the dual of $V$, 
with dense and compact imbeddings.  Let $L$
be the duality mapping   from  $ V $ onto $ V '$, 
and $ g $ be  a linear  continuous symmetric  operator from $  V $ into $  V '$.
Let $I$  denote the identity of $\mathbb R^N$.
We define the diagonal operators
\begin{equation}
\mathcal  L = LI \quad\hbox{and}  \quad\mathcal G =  g   I.\end{equation}
Let  $A$ and $D$ be symmetric and semi-positive definite matrices with constant elements. Consider  the following second order evolution system for the state variable  $U= (u^{(1)},  \cdots ,u^{(N)})^T $:
\begin{equation}
U'' + \mathcal L U    +AU   + D\mathcal GU'  =  0,\label{1.2}\end{equation}
where $`` \ ' \  "$ stands for  the time derivative. 

We first show that if system  \eqref{1.2}  is uniformly stable in the space  $(V \times  H)^N$, then rank$(D) =N$  (see Corollary  \ref{2.10} below).
When  rank$(D)<N$,  system   \eqref{1.2}  is not  uniformly  stable, we then turn to  consider the synchronization. 

For any given integer $p\geqslant  1$, let  \begin{equation}0=n_0< n_1<n_2<\cdots<n_p=N\end{equation}
be  integers such that  $n_{r}-n_{r-1}\geqslant 2 $ for all $r$ with $1\leqslant  r\leqslant  p.$
We
re-arrange the components of the state variable $U$ into $p$ groups
 \begin{equation}(u^{(1)},  \cdots,  u^{(n_1)}),
\ (u^{(n_1+1)}, \cdots, u^{(n_2)}), \cdots\cdots, (u^{(n_{p-1}+1)},
\cdots, u^{(n_p)}).\end{equation}

\begin{definition}   System   \eqref{1.2} is  uniformly  (exponentially) synchronizable by $p$-groups,  if there exist  constants $M\geqslant  1$ and $ \omega>0$,  such that for any given initial data $(U_0, U_1)\in (V \times  H)^N$, the corresponding solution $U$ to system  \eqref{1.2} satisfies
 \begin{align}&\|(u^{(k)}(t)-u^{(l)}(t), {u^{(k)}}'(t)-{u^{(l)}}'(t))\|_{ V \times  H}\label{1.3} \\ \leqslant  &Me^{-\omega t}\|(u^{(k)}_0-u^{(l)}_0, u_1^{(k)}-u_1^{(l)})\|_{ V \times  H},\quad t\geqslant 0\notag
\end{align}
for all $k, l$ with $n_{r-1}+1\leqslant   k, l\leqslant    n_r$ and  all $r$ with $1\leqslant  r\leqslant  p.$
\end{definition}

Now  let us outline the main ideas  in the study of the  uniform synchronization  by $p$-groups.  

Let  $C_p$ be the matrix given  by  \eqref{3.1} below.   Then \eqref{1.3} can be equivalently rewritten as 
\begin{equation}\|C_p(U(t), U'(t))\|_{(  V \times H )^{N-p} }\leqslant   Me^{-\omega t}\|C_p(U_0, U_1)\|_{(  V \times H )^{N-p} }\quad t\geqslant 0.\label{1.4}
 \end{equation}

The matrix $A$ satisfies  the condition of $C_p$-compatibility, if there exists a  symmetric and semi-positive definite matrix $\overline A_p$  such that 
 \begin{align}(C_pC_p^T )^{-1/2}C_pA = \overline A_p(C_pC_p^T )^{-1/2}C_p. \label {1.4n}\end{align} 
Correspondingly,  the reduced matrix $\overline D_p$  can be introduced for $D$ (see Proposition \ref{3.4th}).  Applying $(C_pC_p^T )^{-1/2}C_p$ to \eqref{1.2} and setting  $W= (C_pC_p^T )^{-1/2}C_pU$,  
we   get  a self-closed reduced  system   
\begin{equation}
W''+\mathcal L W +  \overline A_pW    + \overline D_p\mathcal GW'    =  0.
\label{1.5}
\end{equation}
It is clear  that the   uniform  synchronization  by $p$-groups of   system   \eqref{1.2} is equivalent to  the  uniform stability of the reduced system   \eqref{1.5}.     

In Theorem \ref{3.12}, we will show  that under  the condition $\hbox{rank}(D)=N-p$,  if the  scalar equation \begin{equation}
u'' + L u +  g  u' =  0
\label{1.6}
\end{equation}
is uniformly stable in  the space $ V \times  H$, then  system    (\ref{1.2}) is uniformly synchronizable  by $p$-groups.  

Furthermore (see Theorem \ref{3.16}), there exist some functions $u_1,\cdots, u_p$,    
such that
 \begin{align}&\|(u^{(k)}(t)-u_r(t), {u^{(k)}}'(t)-u_r'(t))\|_{ V \times  H}\label{1.7}\\\leqslant  & Me^{-\omega t}
\|(u^{(k)}_0-u^{(l)}_0, u_1^{(k)}-u_1^{(l)})\|_{ V \times  H}, \quad t\geqslant 0
\notag\end{align}
 for all $k, l$ with $n_{r-1}+1\leqslant   k, l\leqslant    n_r$  and  all $r$ with $1\leqslant   r\leqslant   p.$

Moreover,  the functions $u_1,\cdots, u_p$ satisfy  a homogeneous system,
then,   the solution $U$ to system \eqref{1.2} follows a conservative  orbit.  This is quite  different from the approximate boundary synchronization by $p$-groups, since  the approximate boundary  synchronization by $p$-groups in the consensus  sense does not imply that  in the pinning sense in general  (see Chapter 11 in \cite{book}). 

The above approach  is direct and efficient.  The  difficult part of the problem is to show the necessity  of the conditions of  $C_p$-compatibility which are imposed as  physically reasonable hypotheses even for the systems   of ordinary differential equations.  So, we  have to first justify   the necessity  of the conditions of compatibility, 
then, the uniform synchronization  will be studied by means of a serious  mathematical consideration. 

The necessity of the condition of $C_p$-compatibility for $A$, respectively $D$   is  intrinsically  linked with  the rank of the matrix $D$.  We will show (see Proposition \ref{3.10th}) that  $\hbox{rank}(D)\geqslant N-p$ is a necessary condition for the uniform  synchronization by $p$-groups. Then under the minimum rank condition $\hbox{rank}(D)=N-p$, we establish the necessity of the condition of $C_p$-compatibility  for the matrix $A$, respectively $D$ (see Theorem \ref{3.12}).  

Now we  give some  related literatures.  One of the motivation of studying  the synchronization  consists  of establishing the controllability for fewer boundary controls. When the number of  boundary controls  is fewer than the number of state variables, the non-exact  boundary controllability for a coupled system of wave equations with  various  boundary controls in the usual energy space was established in Li and Rao \cite{book}.  However, if  the components of initial data are allowed to  have different levels of energy, then  the exact boundary  controllability for a system of two wave equations   was established  by means of only one  boundary control  in Alabau-Boussouira \cite{Ala1, Ala2}, Liu and Rao \cite{Liu23}, Rosier and de Teresa \cite{Rosier25}.  In \cite{Dem5}, Dehman established the controllability of two coupled wave equations on a compact manifold with only one local distributed control.  In \cite{zuazua24, zuazua3},  Zuazua proposed the average controllability  as  another way to deal with the controllability with fewer controls. The observability inequality is particularly interesting for a trial on the decay rate of  approximate  controllability.   

The paper is organized as follows.   In \S \ref{s2}, we consider  the uniform  stability and  establish a lower bound on the rank of the control matrix, which is 
necessary for the study of the uniform  synchronization.   \S\ref{s3}  is devoted to  the uniform 
synchronization  by $p$-groups.  Under the minimum  rank condition, we show the necessity of the conditions of $C_p$-compatibility for the coupling matrices in the considered system.  
In \S\ref{s4},  we give some examples of applications such as the system  of wave equations with boundary  feedback or (and) locally distributed feedback,  and the system    of  Kirchhoff plate with  distributed feedback.  In \S\ref{s5}, we give some comments on the obtained results and propose  some 
open questions for future development.  

\vskip 1cm

\section {Uniform  stability} \label{s2}
We first recall the following well-posedness  result (see Proposition 3.1 in \cite{rao2}). 
\begin{proposition}  \label{2.0}  System \eqref{1.2}
generates a semi-group of  contractions  with a compact resolvent in the space $(V\times H)^N$. More precisely, for any given initial data $(U_0, U_1)\in (V\times H)^N$, the  corresponding weak solution $U$  to system \eqref{1.2} satisfies 
\begin{equation} U\in C^0(\mathbb R^+, V^N)\cap C^1(\mathbb R^+, H^N)\end{equation} 
and 
\begin{equation}\|(U(t), U'(t))\|_{(V\times H)^N}\leqslant \|(U_0, U_1)\|_{(V\times H)^N},  \quad t\geqslant 0.\end{equation}

\end{proposition}
\begin{definition}   System   \eqref{1.2}  is  uniformly (exponentially) stable  in  the space $ ( V \times  H)^{N}$,  if there exist  constants $M\geqslant  1$ and $\omega>0$, such that for any given initial data $(U_0, U_1)\in ( V  \times  H)^{N}$, the corresponding solution $U$ to system   \eqref{1.2}  satisfies 
 \begin{equation} \|(U(t), U'(t))\|_{( V  \times  H)^{N}}\leqslant
  Me^{-\omega t}\|(U_0, U_1)\|_{( V  \times  H)^{N}},  \quad t\geqslant 0.\label{2.1}\end{equation}\end{definition}
\begin{proposition} \label{2.2} 
Let   $\mathcal R$ be a linear compact  mapping
from $ V $ to $L^2(0, T;  H)$. Then 
we can not find positive constants  $M\geqslant 1$ and $\omega>0$, such that for all $\theta\in  V $, the solution to the following problem \begin{equation} \begin{cases}u ''+ L u=\mathcal R\theta,\\
t=0:\quad u= \theta,\  u'=
0\label{2.3} \end{cases}\end{equation}
satisfies  
 \begin{equation}\|(u(t), u'(t))\|_{ V  \times  H}
 \leqslant   Me^{-\omega t} \|\theta\|_{ V },\quad t \geqslant 0.
\label{2.4}\end{equation}
\end{proposition}
\begin{proof}  
Noting that problem \eqref{2.3} is time invertible,  by well-posedness we have
\begin{equation}\|\theta\|_{ V }\leqslant   \|u(T)\|_{ V }+ \|u'(T)\|_{ H}
+ \int_0^T \|\mathcal R\theta\|_{ H}dt.\label{2.5} \end{equation}
Assume by contradiction that \eqref{2.4} holds for all $\theta \in  V$,  then we have \begin{equation} \|\theta\|_{ V }\leqslant   Me^{-\omega T}\|\theta\|_{ V } +\int_0^T \|\mathcal R\theta\|_{ H}dt.\end{equation}
When $T$ is   large enough, it follows that for all $\theta \in  V$,  we have
\begin{equation}\|\theta\|_{ V }\leqslant   \frac{\sqrt T} {1-Me^{-\omega T}}\|\mathcal R\theta\|_{L^2(0, T;  H)}.\end{equation}
This contradicts the compactness of $\mathcal R$.  The proof is  complete.
\end{proof}

\begin{theoreme} \label{2.5th}  
Let $\widetilde C_q$ be a full row-rank matrix of order $(N-q)\times N$ with
$0\leqslant q<N$.  Assume that  there exist  constants $M\geqslant  1$ and $ \omega>0$,  such that for any given initial data $(U_0, U_1)\in (V \times  H)^{N}$, the corresponding solution $U$ to  system   \eqref{1.2} satisfies 
 \begin{equation}\|\widetilde C_q(U(t), U'(t))\|_{( V  \times  H)^{N-q}}\leqslant
  Me^{-\omega t}\|(U_0, U_1)\|_{( V  \times  H)^{N}},  \quad t\geqslant 0.\label{2.6} \end{equation}
  Then
   \begin{equation}rank(\widetilde C_q D)\geqslant N-q.
   \label{2.7} \end{equation}
\end{theoreme}

\begin{proof}  Assume by contradiction that  the rank condition \eqref{2.7} fails.  Then, we have 
\begin{equation}rank(\widetilde C_q D) = rank(D\widetilde C_q^T ) <   N-q = rank(\widetilde C_q^T ).
\end{equation}
By Proposition 2.11 in \cite{book}, we have 
\begin{equation}Im(\widetilde C_q^T )\cap Ker(D) \not = \{0\}.\end{equation} 
Let $E\in \hbox{Im} (\widetilde C_q^T )$ be a unit vector such that $DE=0.$  
Applying $E$ to system   \eqref{1.2} associated  with the initial data 
\begin{equation}\label{2.8} t= 0:\quad U=\theta E,\quad  U'= 0\end{equation} 
with  $\theta\in V$, and setting  $u = (\!(E, U)\!)$, we get 
\begin{equation}\begin{cases}u ''+ L u=-(\!(E,  AU)\!),\\
t=0:\quad u= \theta,\  u'=
0, \end{cases}\label{2.9} \end{equation}
here and hereafter $(\!(\cdot,  \cdot)\!)$  denotes the inner product with  the associated norm  $\|\cdot\|$  in the euclidian space $\mathbb R^N$.

Now,  we  define  the linear mapping
\begin{equation}\mathcal R:\quad \theta \rightarrow (\!(E,  AU)\!).\end{equation}
Since the matrices $A$  and $D$ are  symmetric  and semi-positive definite,  by  the dissipation of system \eqref{1.2} with the initial data \eqref{2.8},  we have \begin{equation}   \|\mathcal R\theta\|_{L^2(0, T;   V )} + \|\mathcal R\theta\|_{H^1(0, T;   H)}\leqslant   c_T\|\theta\|_V,\end{equation}
where $c_T$ is a positive constant depending only on $T$. 

Noting that the imbedding  from $L^2(0, T;   V ) \cap H^1(0, T;   H)$ into $L^2(0, T;   H)$ is compact (see  Theorem 5.1 in \cite{Lions}),  the mapping $\mathcal R$  is compact from
$ V $ into   $L^2(0, T;   H).$

On the other hand,  noting $E= \widetilde C_q^T x$, we have 
\begin{equation}u = (\!(E, U)\!) = (\!(x, \widetilde C_qU)\!).\end{equation}
Then, it follows from \eqref{2.6} that 
\begin{equation}
\|(u(t), u'(t)) \|_{ V  \times  H} \leqslant c\|\widetilde C_q(U(t), U'(t)) \|_{ V  \times H}\leqslant cM e^{-\omega t}\|\theta \|_{ V },  \quad t\geqslant 0
\end{equation}
for all $\theta\in V$. This contradicts  Proposition \ref{2.2}.  \end{proof}

In particular, taking $\widetilde C_q=I$ in Theorem \ref{2.5th}, we 
get immediately
\begin{corollary}   \label{2.10}  If   system   (\ref{1.2}) is  uniformly   stable, then we 
have rank$(D)= N.$
\end{corollary} 

Conversely,  we have

\begin{theoreme} \label{2.11}    
Assume  that the  scalar equation \begin{equation}u'' + L u +  g  u' =  0\label{2.12} 
\end{equation}
is uniformly   stable in  the space $ V \times  H$.  
If  rank$(D) =N$, then  system    (\ref{1.2})   is uniformly  stable  in the space  $(V \times  H)^N$. \end{theoreme}
\begin{proof}  Following the classical theory (see  \cite{Russell,  Trig}),  the uniform   stability of a semi-group is robust by compact perturbation.  
This property was served in \cite{Asym1997,  rao1} for obtaining  the uniform  stability.   

More precisely, since the mapping $ U\rightarrow  AU$ is compact from $ V$ into  $ H$, the asymptotic stability of  the coupled system    (\ref{1.2}) and the  uniform stability of the  following  decoupled  system  
\begin{equation}
U'' + \mathcal L U + D\mathcal G U'=  0\label{2.13} 
\end{equation}  
yield the  uniform stability  of the coupled system    (\ref{1.2}).  
 
Since rank$(D)=N$, system   \eqref{2.13}  can be  decomposed into  $N$ scalar equations of the same type as  those  in \eqref{2.12}, therefore,  it is uniformly  stable. 

On the other hand,  by Proposition \ref{2.0}, the resolvent of system (\ref{1.2})  is  compact.  Then by the classical theory of  semigroups, the asymptotic stability of   system  (\ref{1.2})  is equivalent to the uniqueness of  the following over-determined  system:
\begin{equation} (\mathcal L+A)\Phi =\beta^2\Phi \quad \hbox{and}\quad  \mathcal G\Phi =0.\end{equation} 
Let $AE= \lambda E$ with $\lambda\geqslant 0$. Then, setting $\phi=(E, \Phi)$, it follows that
\begin{equation} L\phi =(\beta^2-\lambda)\phi \quad \hbox{and}\quad   g \phi =0.\end{equation} 
By the definition of the dual mapping, we have 
$$\langle L\phi, \phi\rangle_{V ',  V} 
=( L\phi, \phi)_{ H}  = \|\phi\|^2_{V}.$$
It follows that 
$\beta^2-\lambda >0$.
Then, we check easily that 
\begin{equation}U = e^{it\sqrt{\beta^2-\lambda}}\phi E\end{equation}
satisfies  system  \eqref{2.13},  which is uniformly stable.  We get  $(\!(E, \Phi)\!)=\phi = 0$ for each eigenvector $E$ of $A$,  then  $\Phi=0$. \end{proof}
\begin{rem} Roughly speaking, Theorem \ref{2.11}  indicates  that the uniform stability of system \eqref{1.2} can be obtained by means of  the scalar equation \eqref{2.12}.  It  provides thus a direct and efficient approach to solve a seemingly  difficult  problem of uniform stability of a complex system.
\end{rem} 

\vskip 1cm

\section{Uniform  synchronization by $p$-groups} \label{s3}
By Corollary \ref{2.10},  when    rank$(D)<N$,    system   \eqref{1.2}  is not   uniformly stable.  
Instead of the stability,  we turn to consider its synchronization by $p$-groups. 

Let $S_r$ be  the full row-rank matrix of order 
$(n_r-n_{r-1}-1)\times (n_r-n_{r-1})$:
 \begin{equation}S_r = \begin{pmatrix}1&-1\\
&1&-1&\cr&&\ddots&\ddots\\
&&&1&-1\end{pmatrix},\quad 1\leqslant   r\leqslant   p. \end{equation}  Define the 
$(N-p)\times N$ matrix $C_p$  of synchronization  by $p$-groups as
\begin{equation}C_p=\begin{pmatrix}
S_1\cr
&S_2\cr &&\ddots\cr
&&&S_p\end{pmatrix}. \label{3.1}\end{equation}
The uniform  synchronization   by $p$-groups (\ref{1.3})
can be equivalently rewritten by \eqref{1.4}, which is easy to be analyzed. 

Let $\epsilon_1,\cdots, \epsilon_N$ be the vectors of the canonical basis of $\rr^N$. Defining  
 \begin{equation} e_r=\sum_{i=n_{r-1}+1}^{n_r}\epsilon_i,\quad  1\leqslant   r\leqslant   p,\label{3.2}\end{equation}
we have  \begin{equation} Ker (C_p)=Span\{e_1,\cdots,
e_p\}.\label{3.3}\end{equation}

\begin{proposition}  (see Proposition 4.2 in \cite{rao2}) The matrix $A$ satisfies the  condition of  $C_p$-compatibility:
\begin{equation}AKer(C_p)\subseteq  Ker(C_p)\label{3.4}\end{equation}
if and only if    there exists  a symmetric and semi-positive definite matrix $\overline A_p$  of order  $(N-p)$, such that \eqref {1.4n} holds.
\end{proposition} 

\begin{proposition}  \label{3.4th}(see Proposition 4.4 in \cite{rao2})  The matrix $D$ satisfies the condition of strong   $C_p$-compatibility: 
\begin{equation}Ker(C_p) \subseteq Ker (D)\label{3.5}\end{equation}     
if and only if  there exists  a  symmetric  and semi-positive definite matrix   $R$ of order $(N-p)$,  such that
\begin{equation}D = C_p^T RC_p. \label{3.6}\end{equation}  
Moreover, setting 
\begin{equation}\label{3.6b}\overline D_p = (C_pC_p^T )^{1/2}R(C_pC_p^T )^{1/2},\end{equation}
we have \begin{equation}(C_pC_p^T )^{-1/2}C_pD= \overline D_p(C_pC_p^T )^{-1/2}C_p.\label{3.7}\end{equation} 
\end{proposition}

\begin{rem}  By  the expression \eqref{3.2}, 
it is easy to check that the condition of $C_p$-compatibility \eqref{3.4}  is equivalent to the row-sum condition by blocks
\begin{align}
 \sum_{j=n_{s-1}+1}^{n_{s}}a_{ij}= \alpha_{rs}, \quad n_{r-1}+1 \leqslant  i \leqslant  n_{r}, \quad  1 \leqslant  r,s \leqslant  p, \end{align}
 where $ \alpha_{rs}$ are some constants.  
In particular, when $p=1$,   $A$ satisfies  the
row-sum condition:
\begin{align}
\sum_{p=1}^N a_{kp}=\alpha,\quad     k=1, \cdots, N.
\end{align}

The condition of strong   $C_p$-compatibility
\eqref{3.5} is equivalen to 
\begin{equation}DKer(C_p)=\{0\}.\end{equation} 
That means that 
$D=(d_{ij})$ satisfies  the null row-sum condition by blocks
\begin{align}
 \sum_{j=n_{s-1}+1}^{n_{s}}d_{ij}= 0, \quad n_{r-1}+1 \leqslant  i \leqslant  n_{r}, \quad  1 \leqslant  r,s \leqslant  p. \end{align}
\end{rem}

Now applying $C_p$ to   system   \eqref{1.2},  and setting $W= C_pU$, we get  a self-closed reduced system \begin{equation}W'' + \mathcal L W  + \overline A_pW + \overline D_{p}GW'  =  0.\label{3.9}\end{equation} 
Moreover, it is easy to check  the following  basic result.
\begin{proposition}  \label{3.9th}
Assume  that  the matrices $A$ and $D$ satisfy the condition of $C_p$-compatibility \eqref{3.4}  and  the condition of strong $C_p$-compatibility \eqref{3.5}, respectively. The uniform   synchronization  by $p$-groups of  system   \eqref{1.2}  in the space 
$(V\times  H)^{N}$
 is equivalent to the uniform   stability of the  reduced  system \eqref{3.9}
in the space  $(V \times  H)^{N-p}$. \end{proposition}
 
Since the reduced matrices $\overline A_p$ and $ \overline D_{p}$ are still  symmetric  and semi-positive definite,  the uniform stability of the reduced  system   (\ref{3.9})
can be treated by Theorem \ref{2.11}.   So,  the  uniform  synchronization  by $p$-groups
is reduced to the uniform  stability. However, the necessity of the condition of $C_p$-compatibility for $A$ and  that of  the  condition of strong $C_p$-compatibility   for $D$  are intrinsically  linked with  the rank of the matrix $D$.   

\begin{proposition}  \label{3.10th} If   system   \eqref{1.2} is  uniformly  synchronizable by $p$-groups, then we  necessarily have
\begin{equation}\hbox{rank}(C_pD)\geqslant   N-p.\end{equation}
\end{proposition}
 
\begin{proof}  It is sufficient to take $\widetilde C_q  = C_p$ in  Theorem \ref{2.5th}.
\end{proof} 

\begin{proposition}  \label{3.10} The following rank condition
\begin{equation}\hbox{rank}(D) = \hbox{rank}(C_pD)=N-p\label{3.11}\end{equation}
holds, if and only if  $Ker(D)$ and $Ker(C_p)$  are bi-orthonormal.
\end{proposition}

\begin{proof} 
By  Proposition 2.11 in \cite{book}, the rank condition \eqref{3.11} is equivalent to
\begin{equation}Ker( D)\cap Im(C_p^T )  = Ker(C_p)\cap Im(D)=\{0\},\end{equation}
namely,
\begin{equation}Ker(D)\cap \{Ker(C_p)\}^\perp  = Ker(C_p)\cap \{Ker(D^T)\}^\perp=\{0\}.\end{equation}
Hence by Proposition 2.5 in \cite{book},  Ker$(D)$ and  Ker$(C_p)$ are bi-orthogonal. 
 \end{proof}
 
\begin{theoreme}   \label{3.12} Assume that  system   (\ref{1.2}) is  uniformly  synchronizable by $p$-groups under the minimal rank conditions (\ref{3.11}).
Then  $A$ satisfies  the condition of $C_p$-compatibility \eqref{3.4} and  $D$ satisfies  the condition of strong $C_p$-compatibility \eqref{3.5}.
\end{theoreme}

\begin{proof}  
Let $U$ be the solution to system \eqref{1.2} with the following initial data:
\begin{equation}  U_0=\sum_{r=1}^pu_{0r}e_r,\quad U_1=\sum_{r=1}^pu_{1r}e_r,\end{equation}
where $u_{0r}\in V$ and $u_{1r}\in H$ for $r=1,\cdots, p$.
Then by \eqref{1.4} we have
\begin{equation}t\geqslant 0: \quad  \|C_p(U(t), U'(t))\|_{( V  \times H)^{N-p}}\leqslant Me^{-\omega t}\|C_p(U_0, U_1)\|_{( V  \times  H)^{N-p}}=0.\end{equation}
There exist some functions $u_1,\cdots, u_p$ in $C^0(\mathbb R^+, V)\cap C^1(\mathbb R^+, H)$, such that
\begin{equation}  U=\sum_{s=1}^pu_se_s.\end{equation}
Then\begin{equation}  \sum_{s=1}^pu_s''e_s  + \sum_{s=1}^pLu_se_s  + \sum_{s=1}^p g  u_s'De_s + \sum_{s=1}^pu_sAe_s = 0.\end{equation}
Applying $C_p$ to  both sides of the above system, it follows that 
\begin{equation} \sum_{s=1}^p g  u_s'C_pDe_s + \sum_{s=1}^pu_sC_pAe_s = 0.\end{equation}
In particular,  by the continuity at $t=0$, we have 
\begin{equation} \sum_{s=1}^p g  u_{1s}C_pDe_s + \sum_{s=1}^pu_{0s}C_pAe_s = 0,\end{equation}
then
\begin{equation}C_pAe_s = 0, \quad C_pDe_s = 0, \quad s=1,\cdots, p.\end{equation}
Thus  $A$ satisfies  the condition of $C_p$-compatibility \eqref{3.4}, and $D$ satisfies a similar condition of $C_p$-compatibility  as  in \eqref{3.4}.

We next show that $D$ satisfies the  condition of strong $C_p$-compatibility \eqref{3.5}. 
In fact, for $s=1,\cdots, p$, we have
\begin{equation}  (\!(De_s, d)\!)  = (\!(e_s, Dd) \!) =0, \quad d\in Ker(D),\end{equation}
then $De_s\in \hbox{Ker}(D)^\perp\cap \hbox{Ker}(C_p)$.
By Proposition \ref{3.10},  Ker$(D)$ is bi-orthogonal to  Ker$(C_p)$, so
$\hbox{Ker}(D)^\perp\cap \hbox{Ker}(C_p) =\{0\}$. Then 
\begin{equation}  De_s =0, \quad s=1,\cdots, p.\end{equation}
We get thus  the condition of strong $C_p$-compatibility \eqref{3.5} for the matrix  $D$. 
\end{proof} 

\begin{theoreme} \label{3.14}   Assume that $A$ satisfies  the condition  of $C_p$-compatibility  \eqref{3.4} and $D$  satisfies the condition  of strong $C_p$-compatibility \eqref{3.6} with  rank$ (R) = N-p.$  Assume  furthermore that the  scalar equation \eqref{2.12} 
is uniformly   stable in  the space $ V \times  H$.  Then  system    (\ref{1.2}) is uniformly synchronizable by $p$-groups in $(  V \times  H )^N$. 
\end{theoreme}

\begin{proof}  By  Proposition \ref{3.9th},  it is sufficient to show  the uniform   stability of the reduced  system \eqref{3.9}. 
By \eqref{3.6b},   $\hbox{rank}(\overline D_{p}) = \hbox{rank}(R) = N-p$.  Then by Theorem \ref{2.11},    the reduced system   \eqref{3.9} is uniformly stable. 
\end{proof}  
 
\begin{theoreme} \label{3.16}   Assume that  system    (\ref{1.2}) is uniformly synchronizable by $p$-groups in $(  V \times H )^N$, then   there exist some functions  $u_1,\cdots, u_p$ in $C^0(\mathbb R^+, V)\cap C^1(\mathbb R^+, H)$ and some positive constants $M\geqslant 1$ and $\omega>0$,
such that setting
\begin{equation}\label{3.17} u=\sum_{r=1}^pu_r
e_r/\|e_r\|,\end{equation}
we have for all $t\geqslant 0$, 
\begin{align}\|(U(t)  - u(t) , 
U' (t)  - u'(t) )\|_{( V  \times H)^N}  \label{3.18} \leqslant Me^{-\omega t}\|C_p(U_0, U_1)\|_{( V  \times  H)^{N-p}}.
\end{align}
Assume furthermore that $A$ satisfies  the condition  of $C_p$-compatibility  \eqref{3.4} and $D$ satisfies  the condition  of strong $C_p$-compatibility \eqref{3.5}.  Then $u$   obeys  a conservative system.
\end{theoreme}
 
\begin{proof}  Let $U$ be the  solution to system \eqref{1.2} with any  given initial data $(U_0, U_1)\in (V\times H)^N$. For $r=1,\cdots ,p$, let $u_r= (\!(U, e_r)\!)/\|e_r\|.$  Noting that 
$\mathbb R^N = \hbox{Ker}(C_p)\oplus \hbox{Im}(C_p^T)$, we have 
\begin{equation}
U =  \sum_{r=1}^pu_re_r/\|e_r\| + C_p^T(C_pC_p^T)^{-1}C_p U =u + C_p^T(C_pC_p^T)^{-1}C_p U.\end{equation}
By \eqref{1.4}, we get 
\begin{align}&\|(U(t)   - u(t) , 
U'(t)  - u'(t) )\|_{( V  \times  H)^N} 
\\ \leqslant   \notag &\|C_p^T(C_pC_p^T)^{-1}\|\|C_p(U(t) , U'(t) )\|_{( V  \times H )^{N-p}} \\\leqslant   \notag &M'e^{-\omega t}\|C_p(U_0, U_1)\|_{( V  \times H )^{N-p}},  \quad t\geqslant 0\end{align}
for some constant $M'\geqslant 1.$ 

Now we will precisely show  the dynamics of the functions $u_1, \cdots, u_p$. First, recall that  the condition  of $C_p$-compatibility \eqref{3.4} implies
 \begin{equation} A e_r = \sum_{s=1}^p\beta_{rs} {\|e_r\|\over \|e_s\|}e_s,\quad r=1, \cdots, p. \label{3.19}\end{equation}
Moreover, since $A$ is symmetric, a straightforward computation shows that 
\begin{equation} (A e_r, e_s)  = \sum_{q=1}^p\beta_{rq} {\|e_r\|\over \|e_q\|}(\!(e_q, e_s)\!)= \beta_{rs} \|e_r\|\|e_s\|\end{equation} and  \begin{equation} (\!(e_r,  Ae_s)\!)= \sum_{q=1}^p\beta_{sq} {\|e_s\|\over \|e_q\|}(\!(e_r,  e_q)\!) = \beta_{sr} \|e_s\| \|e_r\|.\end{equation} It follows that 
\begin{equation}\beta_{rs} = \beta_{sr},\quad 1\leqslant r, s\leqslant p. \notag\end{equation}
On the other hand, the condition  of strong $C_p$-compatibility \eqref{3.5} implies 
 \begin{equation}De_r=0, \quad r=1, \cdots, p.\end{equation}
Then, 
applying $ e_r$  to system \eqref{1.2},  we get the following  conservative system
 \begin{equation} \begin{cases}
u_r''  + Lu_r + \sum_{s=1}^p\beta_{rs}u_s  =0,\\
t=0:\quad u_r= (\!(U_0, e_r)\!)/\|e_r\|,\quad u_r'= (\!(U_1, e_r)\!)/\|e_r\|\end{cases}
 \label{3.20}\end{equation}
for $ r=1, \cdots, p$. 
\end{proof}

\begin{rem}  Classically,  the convergence \eqref{1.3} or equivalently \eqref{1.4}  is called uniform synchronization by $p$-groups in the consensus sense, while the convergence \eqref{1.3}  is in the pinning  sense.  Moreover, the $p$-tuple $u= (u_1, \cdots, u_p)$ is called the uniformly synchronizable state by $p$-groups. 
Theorem   \ref{3.16} indicates that two notions are simply the same. 

Moreover, setting the matrix $B =(\beta_{rs})$, we define  the energy by \begin{equation}\notag E(t) =  \|u(t)\|_{V^p}^2 + (B u(t), u(t))_{H^p} + \|u'(t)\|^2_{H^p}. \end{equation}
Since $B$ is symmetric, we have\begin{equation} E(t)= E(0),  \quad  t \geqslant 0.\label{3.21}\end{equation}Then the orbit of $u$ is lacalized  on  the sphere \eqref{3.21}  which is uniquely determined by the projection of the initial data $(U_0, U_1)$ to Ker$(C_p)$. 

\end{rem}

\begin{rem} 
The condition of strong $C_p$-compatibility \eqref{3.5} implies that (see Proposition 2.13 in \cite{book})
$$	\hbox{rank}(D, AD, \cdots, A^{N-1}D) =   N-p.$$
Following Theorem 4.7  in \cite{rao2},  there does not  exist 
an extended  matrix $\widetilde C_q \ (q<p)$, such that 
$$\widetilde C_q(U(t), U'(t))\rightarrow (0, 0)\quad \hbox{in } 
(V\times H)^N \hbox{ as  }  t\rightarrow +\infty.$$
Unlike  in the case of  the
approximate   boundary synchronization by $p$-groups, there is no  possibility to get any  induced synchronization 
in the present situation (see Chapter 11 in \cite{book}). 
\end{rem}


\vskip 1cm

\section{Applications} \label{s4}

\subsection{Wave equations with boundary feedback}
\setcounter{equation}{0} Let  $\Omega\subset \mathbb R^n$ be
a bounded domain  with  a smooth boundary
$\Gamma= \Gamma_1\cup\Gamma_0$ such that  $\overline \Gamma_1\cap \overline \Gamma_0=\emptyset$ and  mes$(\Gamma_1)>0$.  For fixing idea, we assume that 
mes$(\Gamma_0)>0$.

 Consider the following wave equation
 \begin{equation}\left\lbrace
\begin{array}{ll}
u''-\Delta u=  0 & \mbox{in } \rr^+\times
\Omega,\\
 u=0 &\mbox{on }
\rr^+\times \Gamma_0,
\\
 \partial_\nu u+u'=0 &\mbox{on }
\rr^+\times \Gamma_1,
\end{array}
\right.\label{4.0}
\end{equation}
where $ \partial_\nu$ denotes the outward normal derivative on the boundary. The uniform  stability of \eqref{4.0} was abundantly studied  by different approaches in the literature, we only quote  \cite{Asmp1993, lasiecka, lebau} and the references therein.  

Now, let $A$ and $D$ be symmetric and semi-positive definite matrices of order $N$. We consider  the  following  system   of  wave equations:
\begin{equation}\left\lbrace
\begin{array}{ll}
U''-\Delta U + AU  =  0 &\mbox{in }\rr^+\times
\Omega,\\
U=0 &\mbox{on }
\rr^+\times \Gamma_0,
\\
 \partial_\nu U+DU'=0&\mbox{on }
\rr^+\times \Gamma_1.
\label{4.1}
\end{array}
\right.
\end{equation}

Let $H^1_{\Gamma_0}(\Omega)$ denote the subspace of $H^1(\Omega)$, composed  of   functions with vanishing trace on $\Gamma_0$. 
Multiplying (\ref{4.1}) by $\Phi\in H^1_{\Gamma_0}(\Omega)$ and integrating by parts, we get the following variational formulation:
\begin{equation}
\int_\Omega(\!(U'', \Phi)\!)dx  +\int_\Omega(\!(\nabla U, \nabla \Phi)\!)dx  + \int_\Omega (\!(AU, \Phi)\!)dx+ \int_{\Gamma_1}(\!(DU', \Phi)\!)d \Gamma=0. \label{4.2}
\end{equation}
Define
\begin{equation}
\langle Lu, \phi\rangle = \int_\Omega \nabla u\cdot\nabla \phi dx,\quad
\langle  g  v, \phi\rangle=\int_{\Gamma_1} v\phi d\Gamma. 
\end{equation}
Then (\ref{4.2})  can be rewritten as  
\begin{equation}
U''+ \mathcal LU +AU+D\mathcal GU'=0.\label{4.3}
\end{equation}

Moreover, since  the scalar equation  \eqref{4.0} is  uniformly  stable  in $H^1_{\Gamma_0}(\Omega)\times L^2(\Omega)$,  applying   
Theorem \ref{3.14} and Theorem \ref{3.16},  we immediately obtain the following 
   
\begin{theoreme}\label{4.4}  Assume that $A$ satisfies  the condition  of $C_p$-compatibility  \eqref{3.4} and $D$  the condition  of strong $C_p$-compatibility \eqref{3.6} with  rank$(R) = N-p.$  Then  the  system   of wave equations   (\ref{4.1})  is uniformly synchronizable by $p$-groups in the space $(H^1_{\Gamma_0}(\Omega)\times L^2(\Omega))^N$.

Moreover, for any 
given initial data $(U_0, U_1)\in (H^1_{\Gamma_0}(\Omega)\times L^2(\Omega))^N$,  consider the problem \begin{equation}\left\lbrace
\begin{array}{ll}
u_r''-\Delta u_r +\sum_{s=1}^p\beta_{rs}u_s=  0 & \mbox{in } \rr^+\times
\Omega,\\
 u_r=0 &\mbox{on }
\rr^+\times \Gamma_0,\\
 \partial_\nu u_r=0 &\mbox{on }
\rr^+\times \Gamma_1,\\
t=0:\quad u_r= (\!(U_0, e_r)\!)/\|e_r\|, \quad   u_r'= (\!(U_0, e_r)\!)/\|e_r\| &\mbox{in }
 \Omega
\end{array}
\right.
\end{equation}
for $r=1,\cdots, p$,  and the coefficients $\beta_{rs}$ are given by  \eqref{3.19}.
Then, setting $u=\sum_{r=1}^pu_re_r/\|e_r\|$, the corresponding solution $U$ to system   (\ref{4.1}) 
satisfies \begin{align} &\|(U(t)  - u(t) , 
U'(t)  - u'(t) )\|_{(H^1_{\Gamma_0}(\Omega)\times L^2(\Omega))^N}  \\\leqslant &Me^{-\omega t}\|C_p(U_0, U_1)\|_{(H^1_{\Gamma_0}(\Omega)\times L^2(\Omega))^{N-p}},  \quad t\geqslant 0.\notag
\end{align}
\end{theoreme}

\medskip 

\subsection{Wave equations with locally distributed feedback} 
Let $\Omega\subset \mathbb R^n$ denote a bounded domain with smooth boundary $\Gamma$.  Let  $\omega\subset \Omega$  denote the damped domain.   

Let $a$ be a smooth  function  such that
\begin{equation}a(x)\geqslant 0, \quad \forall x\in \Omega \quad\hbox{and}  \quad a(x)\geqslant a_0>0,\quad \forall x\in \omega.\label{4.5}\end{equation}

Consider the uniform stability of the  following  locally damped  scalar  wave  system \begin{equation}\left\lbrace
\begin{array}{ll}
u''-\Delta u  +au'=  0 & \mbox{in } \rr^+\times
\Omega,\\
u=0& \mbox{on }
\rr^+\times \Gamma.
\end{array}
\right.\label{4.5b}
\end{equation}

This is a very challenge and promising  issue. There is a large amount of literatures that we will comment briefly.
The uniform  decay was first established  by multipliers in \cite{haraux}  as  $\omega$ is a neighbourhood of the boundary.  
Later, the result was generalized  in  \cite{zuazua1}  to semi-linear case.    
When $\Omega$ is a compact Riemann manifold without boundary and $\omega$ satisfies  the  geometric optic condition, the uniform stability was established  by  a micro-local approach in  \cite{rauch}.  

Now, consider  the  following  system   of locally damped wave equations:
\begin{equation}\left\lbrace
\begin{array}{ll}
U''-\Delta U + AU   + aDU'=  0 & \mbox{in }  \rr^+\times
\Omega,\\
U=0&\mbox{on }
\rr^+\times \Gamma,
\label {4.6}
\end{array}
\right.
\end{equation}
where $A$ and $D$ are symmetric and semi-positive definite matrices with constant elements. Multiplying  system   (\ref {4.6}) by $\Phi \in  H^1_0(\Omega)$ and integrating by parts, we get the following variational formulation:
\begin{equation}
\int_{\Omega}(\!(U'', \Phi)\!)dx +\int_{\Omega}(\!(\nabla U,\nabla \Phi)\!)dx +\int_{\Omega} (\!(AU, \Phi)\!) dx+\int_{\Omega} a(\!(DU', \Phi)\!)d\Gamma=0.\label {4.7}
\end{equation}
Let $L$ and $ g $  be  the linear continuous  mappings   from $H^1_0(\Omega)$ into $H^{-1}(\Omega)$, defined  by
\begin{equation}
\langle Lu, \phi\rangle = \int_\Omega \nabla u\cdot \nabla \phi dx\quad\hbox{and}\quad 
\langle  g  v, \phi\rangle =\int_\Omega av \phi dx,
 \end{equation}
 respectively.  Then the variational problem (\ref {4.7})  can be rewritten as  
\begin{equation}
U''+\mathcal LU +AU+D\mathcal GU'=0.\label{4.9}
\end{equation}

Then, applying Theorem \ref{3.14} and Theorem \ref{3.16},  we have
\begin{theoreme}  Assume that  the damped domain $\omega\subset \Omega$ contains  a neighbourhood of the whole boundary $\Gamma$.  Assume furthermore that $A$ satisfies  the condition  of $C_p$-compatibility  \eqref{3.4} and $D$  the condition  of strong $C_p$-compatibility \eqref{3.6} with  rank$(R) = N-p.$  Then   system    (\ref{4.9}) is uniformly synchronizable by $p$-groups in the space $(H^1_0(\Omega)\times L^2(\Omega))^N$. 

Moreover, for any 
given initial data $(U_0, U_1)\in (H^1_0(\Omega)\times L^2(\Omega))^N$,  consider the problem \begin{equation}\left\lbrace
\begin{array}{ll}
u_r''-\Delta u_r +\sum_{s=1}^p\beta_{rs}u_s=  0 & \mbox{in } \rr^+\times
\Omega,\\
 u_r=0 &\mbox{on }
\rr^+\times \Gamma,\\
t=0:\quad u_r= (\!(U_0, e_r)\!)/\|e_r\|, \quad  u_r'= (\!(U_0, e_r)\!)/\|e_r\| &\mbox{in }
 \Omega
\end{array}
\right.
\end{equation}
for $r=1, \cdots, p$,  and  the coefficients $\beta_{rs}$ are given by  \eqref{3.19}.
Then, setting $u=\sum_{r=1}^pu_re_r/\|e_r\|$, the corresponding solution $U$ to system   (\ref{4.9}) 
satisfies \begin{align}&\|(U(t)  - u(t), 
U' (t) - u'(t))\|_{(H^1(\Omega)\times L^2(\Omega))^N}  \\ \leqslant &Me^{-\omega t}\|C_p(U_0, U_1)\|_{(H^1(\Omega)\times L^2(\Omega))^{N-p}}, \quad t\geqslant 0.\notag\end{align}
\end{theoreme}

\medskip

\subsection{Kirchhoff plate equations with locally distributed feedback}  In this sub-section $\Omega$ is a bounded domain in $\mathbb R^2$, occupied by an elastic thin plate.  We refer  to \cite{lagnese} for the stabilization of  linear models.

Let $a$ be a smooth and non-negative function such that \eqref{4.5} holds. Assume that $\omega$ contains a neighbourhood of the whole boundary $\Gamma$.  Then,   the  following system of  plate equation  
 \begin{equation}\left\lbrace
\begin{array}{ll}
u''+\Delta^2 u  +a u'=  0&   \mbox{in }
\rr^+ \times \Omega,\\
u=\partial_\nu u =0&\mbox{on } 
\rr^+\times \Gamma
\label{4.10}
\end{array}
\right.
\end{equation}
 is  uniformly  stable  in $H^2_0(\Omega)\times L^2(\Omega)$ (see \cite{lagnese} for details). 
 
Consider the following  system  :
 \begin{equation}\left\lbrace
\begin{array}{ll}
U''+\Delta^2 U + AU  +  a DU'=  0&   \mbox{in }
\rr^+ \times \Omega,\\
U=\partial_\nu U =0&\mbox{on } 
\rr^+\times \Gamma,
\label{4.11}
\end{array}
\right.
\end{equation}
where $A$ and $D$ are symmetric and semi-positive definite matrices with constant elements.  
Multiplying  system   (\ref{4.11}) by $\Phi\in H^2_0(\Omega)$ and integrating by parts, we get the following variational formulation:\begin{equation}
\int_{\Omega}(\!(U'', \Phi )\!)dx  + \int_{\Omega} (\!(\Delta U,  \Delta\Phi)\!) dx+\int_{\Omega} (\!(AU, \Phi)\!) dx+
\int_{\Omega}a(\!(DU', \Phi)\!)dx=0.\label {4.12}
\end{equation}
Let $L$ and $  g $  be  defined  by
\begin{equation}
\langle Lu, \phi\rangle = 
\int_{\Omega} \Delta u \Delta\phi dx\quad\hbox{and}\quad  
\langle  g  v, \phi\rangle=
\int_{\Omega}av\phi dx,
\end{equation}respectively.  (\ref{4.12}) can be interpreted as
\begin{equation}
U''+\mathcal  LU +AU+D\mathcal GU'=0. 
\end{equation}
Then,  applying Theorem \ref{3.14} and Theorem \ref{3.16},  we have

\begin{theoreme}
Assume that $A$ satisfies  the condition  of $C_p$-compatibility  \eqref{3.4} and $D$  the condition  of strong $C_p$-compatibility \eqref{3.6} with  rank$(R) = N-p.$ 
Then  system   (\ref{4.11}) is  uniformly  synchronizable by $p$-groups in $(H^2_0(\Omega)\times L^2(\Omega))^N$.

Moreover, for any 
given initial data $(U_0, U_1)\in (H^2_0(\Omega)\times L^2(\Omega))^N$, consider the problem \begin{equation}\left\lbrace
\begin{array}{ll}
u_r'' +\Delta^2 u_r +\sum_{s=1}^p\beta_{rs}u_s=  0 & \mbox{in } \rr^+\times
\Omega,\\
 u_r = \partial_\nu u_r=0 &\mbox{on }
\rr^+\times \Gamma,\\
t=0:\quad u_r= (\!(U_0, e_r)\!)/\|e_r\|, \quad  u_r'= (\!(U_0, e_r)\!)/\|e_r\|&\mbox{in }
 \Omega
\end{array}
\right.
\end{equation}
for $r=1, \cdots, p$,  and the coefficients $\beta_{rs}$ are given by  \eqref{3.19}.
Then, setting $u=\sum_{r=1}^pu_re_r/\|e_r\|$, the corresponding solution $U$ to system   (\ref{4.11}) 
satisfies \begin{align} &\|(U(t)   - u(t) , 
U' (t)  - u'(t) )\|_{(H^2(\Omega)\times L^2(\Omega))^N}  \\ \leqslant &Me^{-\omega t}\|C_p(U_0, U_1)\|_{(H^2(\Omega)\times L^2(\Omega))^{N-p}},  \quad t\geqslant 0.\notag
\end{align}
\end{theoreme}

\begin{rem}  The above examples are classic and illustrate the applications of the abstract theory.  In fact,  Theorems 3.8 and 3.9  are also applicable for many other models, such as system of wave equations with viscoelastic (Kelvin-Voigt) damping,  system of Kirchhoff plate equations with boundary shear force and bending moment damping etc.

\end{rem}
\vskip 1cm

\section{Perspective comments.} \label{s5} Up to now, we have started the work on a simplified model with only one damping. 
Many related problems can be considered later. 

\medskip

(i) By the definition  of uniform synchronization by $p$-groups:
 \begin{equation} \|C_p(U(t), U'(t))\|_{( V  \times H )^{N-p} }\leqslant   Me^{-\omega t}\|C_p(U_0, U_1)\|_{( V  \times H )^{N-p} }, \quad t\geqslant 0,\label{5.0} 
 \end{equation}
if $C_p(U_0,  U_1)=(0, 0)$, then 
 \begin{equation}C_pU(t)\equiv 0, \quad  t\geqslant 0. \end{equation} 
Thus, for any given synchronized initial data, the solution 
 is always  synchronized.  This simplifies much the study on the 
 necessity of the conditions of $C_p$-compatibility given in Theorem \ref{3.12}. 
 
 A more natural definition of uniform synchronization by $p$-groups should  be given by 
 \begin{equation}\|C_p(U(t), U'(t))\|_{( V  \times H )^{N-p} }\leqslant   Me^{-\omega t}\|(U_0, U_1)\|_{( V  \times H )^N }, \quad t\geqslant 0.\label{5.1}
 \end{equation}
In this case, the  solution is not  automatically synchronized  even for the synchronized initial data. The  situation will be chaotic and  presents  certainly many interesting questions. 

\medskip

(ii) Instead of the uniform decay rate  given by \eqref{5.0}, we can  consider the polynomial  decay rate as 
\begin{equation}\|C_p(U(t), U'(t))\|_{( V  \times H )^{N-p} } = O( (1+ t)^{-\delta}),  \quad t\geqslant 0,
\label{5.2}
 \end{equation} with a positive power $\delta$.  We refer to \cite{Rao2020, Rao2019} and the references therein for the recent progress on the polynomial stability of indirectly damped wave equations.   

\medskip

(iii) We may consider a  system  with several damping  of different types: 
\begin{equation}
U'' + \mathcal L U    +AU  + D_1\mathcal G_1U'  +  D_2\mathcal G_2U' =  0, \label{5.3}\end{equation}
 where $\mathcal G_1$ and $\mathcal G_2$  can be internal and boundary damping  for wave equations, and bending moment and shear  force damping  for plate equations, respectively. Many related questions can be asked, for example:
 
(a) Let $D= (D_1, D_2)$ be the composite  damping matrix.  Is  Kalman's rank condition on $(A, D)$ still  sufficient  for the  asymptotic stability as what has been done in \cite{rao2}?
 
(b) Is the condition rank$(D) = N$  still sufficient  for the  uniform stability as we have 
done in the present work? 

The main difficulty  comes from the interaction of the numerous  matrices
$A, D_1, D_2$,   somewhat like for coupled Robin problem in \cite{luxing}.  
The key idea is to 
separate them as the coupling terms  are compact, so more regularity  seems to be necessary. 
 
We do not have any answer yet for each  question,  but the first attempt already shows  some interesting  results  for developing  the research  in these directions. 

\vskip1cm

{\bf Acknowledgement } This work is supported by  
National Natural  Sciences  Foundation of China under Grant 11831011.

\end{document}